\documentclass[12pt]{article}
\usepackage{amssymb}
\usepackage{amsmath,amsthm}
\usepackage[latin1]{inputenc}
\usepackage{hyperref}
\usepackage{color}
\usepackage{graphicx}
\hypersetup{colorlinks=true, linkcolor=blue, citecolor=blue,
urlcolor=blue}

\newtheorem{remark}{Remark}

\newtheorem{lemma}[remark]{Lemma}
\newtheorem{theorem}[remark]{Theorem}

\title{Corrections to the article ``The
metric dimension of graph with pendant edges" [Journal of Combinatorial Mathematics and Combinatorial Computing,
65 (2008) 139--145]}

\author{D. Kuziak$^{1}$,  J. A.
Rodr\'{\i}guez-Vel\'{a}zquez$^{2}$ and I. G.
Yero$^{2}$\\
    \\
$^1${\small Faculty of Applied Physics and  Mathematics}\\
{\small Gda\'nsk University of Technology,} {\small
ul. Narutowicza 11/12 80-233 Gda\'nsk, Poland} \\ {\small
dkuziak\@@mif.pg.gda.pl}
\\
\\
$^2${\small Departament d'Enginyeria Inform\`atica i Matem\`atiques,}\\
{\small Universitat Rovira i Virgili,}  {\small Av. Pa\"{\i}sos
Catalans 26, 43007 Tarragona, Spain.} \\{\small
ismael.gonzalez\@@urv.cat, juanalberto.rodriguez\@@urv.cat}
}

\begin{document}

\maketitle

\begin{abstract}
We show that the principal results of the article ``The
metric dimension of graph with pendant edges" [Journal of
Combinatorial Mathematics and Combinatorial Computing, 65 (2008)
139--145] do not hold. In this paper we correct the results and we
solve two open problems described in the above mentioned paper.
\end{abstract}



\section{Introduction}
 Let $G=(V,E)$ be a simple graph of order $n=|V|$.
Let $u,v\in V$ be two different vertices of $G$, the distance $d(u,v)$ between vertices $u$ and $v$ is the length
of the shortest path between $u$ and $v$. Given a set of vertices $S=\{v_1,v_2,...,v_k\}$ of $G$,
the {\it metric representation} of a vertex $v\in V$ with respect to $S$ is the vector $r(v|S)=(d(v,v_1),d(v,v_2),...,d(v,v_k))$.
We say that $S$ is a {\it resolving set} for $G$ if for every pair of different vertices $u,v\in V$, $r(u|S)\ne r(v|S)$.
The {\it metric dimension} of $G$ is the minimum cardinality of any resolving set for $G$ and it is denoted by $dim(G)$.
The concept of metric dimension was introduced first independently by Harary and Melter \cite{harary} and Slater \cite{leaves-trees}, respectively.

Let $G$ and $H$ be two graphs of order $n$ and $m$, respectively. The corona
product $G\odot H$ is defined as the graph obtained from $G$ and $H$ by taking one copy of $G$ and $n$ copies of $H$ and
then joining by edges all the vertices from the $i^{th}$-copy of $H$ with the $i^{th}$-vertex of $G$.

Given the graphs $G$ and $H$ with set of vertices
$V_1=\{v_1,v_2,...,v_{n}\}$ and $V_2=\{u_1,u_2,...,u_{m}\}$,
respectively, the Cartesian product of $G$ and $H$ is the graph
$G\times H$ formed by the vertices $V=\{(v_i,u_j)\;:\;1\le i\le
n,\,1\le j\le m\}$ and two vertices $(v_i,u_j)$ and $(v_k,u_l)$ are
adjacent in $G\times H$ if and only if ($v_i=v_k$ and $u_j\sim u_l$)
or ($v_i\sim v_k$ and $u_j=u_l$). The metric dimension of Cartesian
product graph is studied in \cite{pelayo1}.

 The following results related to the metric
dimension of the graph $(P_n\times P_m)\odot K_1$ and $(K_n\times P_m)\odot K_1$ were published in \cite{baskoro}.

\begin{theorem}{\em \cite{baskoro}} For $n\ge 1$ and $1\le
m\le 2$, $dim((P_n\times P_m)\odot K_1)=2$.
\end{theorem}

\begin{theorem}{\em \cite{baskoro}} For $n\ge 3$,
$$dim((K_n\times P_m)\odot K_1)=\left\{\begin{array}{ll}
                                        n-1, & m=1, \\
                                        n, & m=2.
                                      \end{array}
\right.$$
\end{theorem}

In this paper we show that the above results
are not correct for the case $m=2$ and $n\ge 3$. We also solve the
general case $m\ge 2$.

\section{Results}

\begin{theorem}
If $n\ge 3$ and $m\ge 2$, then $dim((P_{n}\times P_{m})\odot
K_1)=3$.
\end{theorem}

\begin{proof}
Let $\{v_1,v_2,...,v_n\}$ and $\{u_1,u_2,...,u_m\}$ be the set of vertices of the graphs $P_{n}$ and $P_{m}$,
respectively.  The vertices of $P_{n}\times P_{m}$ will be denoted by  $v_{ij}=(v_i,u_j)$ and the pendant vertex of
$v_{ij}$ in  $(P_{n}\times P_{m})\odot K_1$ will be denoted by $u_{ij}$. 
We will show that $S=\{v_{11},v_{1m},v_{nm}\}$ is a
resolving set for $(P_{n}\times P_{m})\odot K_1$. The representations
of vertices of $(P_{n}\times P_{m})\odot K_1$ with respect
to $S$ are given by the following expressions,
\begin{align*}
r(v_{ij}|S)&=(d(v_{ij},v_{11}),d(v_{ij},v_{1m}),d(v_{ij},v_{nm}))\\
&=(i+j-2,m+i-j-1,m+n-i-j),
\end{align*}
\begin{align*}
r(u_{ij}|S)&=(d(u_{ij},v_{11}),d(u_{ij},v_{1m}),d(u_{ij},v_{nm}))\\
&=(i+j-1,m+i-j,m+n-i-j+1).
\end{align*}
Now, let us suppose there exist two different vertices $x,y$ of
$(P_{n}\times P_{m})\odot K_1$ such that $r(x|S)=r(y|S)$. If
$x=v_{ij}$ and $y=v_{kl}$, then $i\ne k$ or $j\ne l$ and we obtain
that
$$(i+j-2,m+i-j-1,m+n-i-j)=(k+l-2,m+k-l-1,m+n-k-l).$$
Which leads to $i=k$ and $j=l$, a contradiction. Analogously we
obtain a contradiction if $x=u_{ij}$ and $y=u_{kl}$. On the other
hand, if $x=v_{ij}$ and $y=u_{kl}$, then we have
$$(i+j-2,m+i-j-1,m+n-i-j)=(k+l-1,m+k-l,m+n-k-l+1),$$ which leads to
 $1=-1$, a contradiction. So,
for every different vertices $x,y$ of $(P_{n}\times P_{m})\odot
K_1$, we have $r(x|S)\ne r(y|S)$. Therefore, $dim((P_{n}\times
P_{m})\odot K_1)\le 3$.

On the other hand, since $(P_{n}\times
P_{m})\odot K_1$ is not a path, $dim((P_{n}\times
P_{m})\odot K_1)\ge 2$. Now let us suppose $S'=\{a,b\}$ is a resolving set for
$(P_n\times P_{m})\odot K_1$. If there exist two different paths of
length $d(a,b)$ between $a$ and $b$, then there exist two different
vertices $c,d$ of $(P_n\times P_{m})\odot K_1$ such that
$d(c,a)=d(d,a)$ and $d(c,b)=d(d,b)$, a contradiction. Let us suppose there is
only one path $Q$, of length $d(a,b)$, between $a$ and $b$. Thus,
all the vertices of $Q$, except possibly $a$ or $b$ which could be
pendant vertices, belong either to a copy of $P_n$ or to a copy of
$P_m$. We consider the following cases.

Case 1: If every vertex belonging to the path $Q$ has degree less or
equal than three, then $m=2$ and $S'\subset
\{u_{11},v_{11},u_{21},v_{21}\}$ or $S'\subset
\{u_{1n},v_{1n},u_{2n},v_{2n}\}$. Let us suppose $S'\subset
\{u_{11},v_{11},u_{21},v_{21}\}$. Now, for the vertices
$u_{1i},v_{1,i+1}$, $2\le i\le n-1$ we have that
\begin{align*}
d(u_{i1},a)&=d(u_{i1},v_{11})+d(v_{11},a)\\
&=d(v_{i+1,1},v_{11})+d(v_{11},a)\\
&=d(v_{i+1,1},a),
\end{align*}
\begin{align*}
d(u_{i1},b)&=d(u_{i1},v_{11})+d(v_{11},b)\\
&=d(v_{i+1,1},v_{11})+d(v_{11},b)\\
&=d(v_{i+1,1},b).
\end{align*}
Thus, $r(u_{i1}| S')=r(v_{i+1,1}|S')$, a contradiction.
On the contrary, if $S'\subset
\{u_{1n},v_{1n},u_{2n},v_{2n}\}$, then for the vertices
$u_{i1},v_{i-1,1}$, $2\le i\le n-1$ we have
\begin{align*}
d(u_{i1},a)&=d(u_{i1},v_{1n})+d(v_{1n},a)\\
&=d(v_{i-1,1},v_{1n})+d(v_{1n},a)\\
&=d(v_{i-1,1},a),
\end{align*}
\begin{align*}
d(u_{i1},b)&=d(u_{i1},v_{1n})+d(v_{1n},b)\\
&=d(v_{i-1,1},v_{1n})+d(v_{1n},b)\\
&=d(v_{i-1,1},b).
\end{align*}
Thus, $r(u_{i1}| S')=r(v_{i-1,1}|S')$, a contradiction.

Case 2: There exists a vertex $v$ of degree four belonging to the path $Q$. So, $v$ has two neighbors $c,d$ not
belonging to $Q$, such that $d(c,a)=1+d(v,a)=d(d,a)$ and
$d(c,b)=1+d(v,b)=d(d,b)$. Thus, $r(c| S')=r(d|S')$, a contradiction.
Hence,
$dim((P_{n}\times P_{m})\odot
K_1)\ge 3$. Therefore, the result follows.
\end{proof}

The
following lemmas are useful to obtain
the next result.

\begin{lemma}{\em \cite{pelayo1}}\label{remark-kn-pm}$\,$
If $n\ge 3$ then $dim(K_n\times P_m)=n-1$.
\end{lemma}

\begin{lemma}{\em \cite{buczkowski}}\label{one-pendant-edge}
If $G_1$ is a graph obtained by adding a pendant edge to a
nontrivial connected graph $G$, then
$$dim(G)\le dim(G_1)\le dim(G) + 1.$$
\end{lemma}

\begin{theorem}
If $m\ge 2$, then
$$dim((K_{n}\times P_{m})\odot K_1)=\left\{\begin{array}{cc}
    n-1, & \textrm{for $n\ge 4$,} \\
    3, & \textrm{for $n=3$.}
  \end{array}\right.
$$
\end{theorem}

\begin{proof}
Similarly to the above proof, let $v_{ij}=(v_i,u_j)$ be the set of
vertices of $K_{n}\times P_{m}$, where $v_i$, $1\le i\le n$ and
$u_j$, $1\le j\le m$ are vertices of the graphs $K_{n}$ and $P_{m}$,
respectively. Let us  denote by $u_{ij}$ the pendant vertex of
$v_{ij}$. Assume that $n=3$. We will 
show that $S=\{v_{11},v_{21},v_{3m}\}$ is a resolving set for $(K_3\times P_m)\odot K_1$. Let
us consider two different vertices $x,y$ of $(K_3\times P_{m})\odot
K_1$. We have the following cases.

Case 1: $x=v_{ij}$ and $y=v_{kl}$. If $j=l$, then $i\ne k$ and
either $i\ne 3$ or $k\ne 3$, say $i\ne 3$. So, for $v_{i1}\in S$ we
have $d(x,v_{i1})=j-1<j=d(y,v_{i1})$. On the contrary, say $j<l$. If
$i\ne 3$ or $k\ne 3$, for instance, $i\ne 3$, then for $v_{i1}\in S$
we have $d(x,v_{i1})=j-1<l-1\le d(y,v_{i1})$. Now, if $i=k=3$, then
$d(x,v_{3m})=m-j>m-l=d(y,v_{3m})$.

Case 2: $x=u_{ij}$ and $y=u_{kl}$. Is analogous to the above case.

Case 3: $x=v_{ij}$ and $y=u_{kl}$. If $j=l$ and $i=k=3$, then we
have $d(x,v_{3m})=m-j<m-j+1=d(y,v_{3m})$. Also, if  $j=l$ and ($i\ne
3$ or $k\ne 3$), say $i\ne 3$, then for $v_{i1}\in S$ we have
$d(x,v_{i1})=j-1<j\le d(y,v_{i1})$. On the other hand, if $j\ne l$,
we consider the following subcases.

Subcase 3.1: $i=k$ and $i\ne 3$. If $j=l+1$, then we have that
$d(x,v_{3m})=m-j+1=m-l<m-l+2=d(y,v_{3m})$. On the other hand, if
$j\ne l+1$, then for $v_{i1}\in S$ we have $d(x,v_{i1})=j-1\ne
l=d(y,v_{i1})$.

Subcase 3.2: $i=k=3$. If $j=l-1$, then there exists $v_{r1}\in S$,
$r\ne 3$ such that $d(x,v_{r1})=j=l-1<l+1=d(y,v_{r1})$. On the other
hand, if $j\ne l-1$, then we have that $d(x,v_{3m})=m-j\ne
m-l+1=d(y,v_{3m})$.

Subcase 3.3: $i\ne k$. Hence, we have either $i\ne 3$ or $k\ne 3$,
for instance $i\ne 3$. If $d(x,v_{i1})=j-1=d(y,v_{i1})$, then there
exist $v_{r1}\in S-\{v_{i1}\}$, $r\ne 3$, such that
$d(x,v_{r1})=j>j-1\ge d(y,v_{r1})$.

Therefore, $dim((K_3\times P_{m})\odot K_1)\le 3$.

On the other hand, let $S'=\{a,b\}$ be a resolving set for  $(K_3\times
P_{m})\odot K_1$. If
there exist two different paths of length $d(a,b)$ between $a$ and
$b$, then there exist two different vertices $c,d$ of $(K_3\times
P_{m})\odot K_1$ such that $d(c,a)=d(d,a)$ and $d(c,b)=d(d,b)$.  Hence, $r(c| S')=r(d|S')$, a contradiction. Moreover,
if there is only one path $Q$, of length $d(a,b)$,
between $a$ and $b$, then
there exists a vertex $v$ of degree four belonging to the path $Q$. So, $v$ has two neighbors $c,d$ not
belonging to $Q$, such that $d(c,a)=1+d(v,a)=d(d,a)$ and
$d(c,b)=1+d(v,b)=d(d,b)$. Thus, $r(c| S')=r(d|S')$, a contradiction.
Thus, $dim((K_3\times P_{m})\odot K_1)\ge 3$. Therefore, for $n=3$, the
result follows.

Now, let $n\ge 4$.  We will show that $S=\{v_{1m},v_{31},v_{41},...,v_{n1}\}$ is a
resolving set for $(K_{n}\times P_{m})\odot K_1$. Let us consider two
different vertices $x,y$ of $(K_{n}\times P_{m})\odot K_1$. We have
the following cases.

Case 1: $x=v_{ij}$ and $y=v_{kl}$. If $j=l$, then $i\ne k$. Let us
suppose $i=1$ and $k=2$. Hence for $v_{1,m}\in S$ we have
$d(x,v_{1m})=m-j<m-j+1=d(y,v_{1m})$. Now, if $i\notin \{1,2\}$ or
$k\notin \{1,2\}$, then we have $v_{i1}\in S$ or $v_{k1}\in S$, say
$v_{i1}\in S$. Thus, we have $d(x,v_{i1})=j-1<j=l=d(y,v_{i1})$.

On the other hand, if $j\ne l$, say $j<l$, then there exists
$v_{t1}\in S$, $t\in \{3,...,n\}$, $t\ne k$, such that
\begin{align*}
d(x,v_{t1})&=d(x,v_{i1})+d(v_{i1},v_{t1})\\
&\le j-1+d(v_{k1},v_{t1})\\
&< l-1+d(v_{k1},v_{t1})\\
&=d(y,v_{k1})+d(v_{k1},v_{t1})\\
&=d(y,v_{t1}).
\end{align*}

Case 2: $x=u_{ij}$ and $y=u_{kl}$. Since $d(u_{ij},v)=d(v_{ij},v)+1$
for every $v\in S$, we proceed analogously to the above case and we
obtain that $r(u_{ij}|S)\ne r(u_{kl}|S)$.

Case 3: $x=v_{ij}$ and $y=u_{kl}$.   If $j\le l$, then for every $v_{t1}\in S$ we have
\begin{align*}
d(x,v_{t1})&=d(x,v_{i1})+d(v_{i1},v_{t1})\\
&=j-1+d(v_{i1},v_{t1})\\
&<l-1+d(v_{i1},v_{t1})\\
&\le l+d(v_{k1},v_{t1})\\
&=d(y,v_{k1})+d(v_{k1},v_{t1})\\
&=d(y,v_{t1}).
\end{align*}
Now, if $j>l$, then we have
\begin{align*}
d(x,v_{1m})&=d(x,v_{im})+d(v_{im},v_{1m})\\
&=m-j+d(v_{im},v_{1m})\\
&<m-l+d(v_{im},v_{1m})\\
&\le m-l+1+d(v_{km},v_{1m})\\
&=d(y,v_{km})+d(v_{km},v_{1m})\\
&=d(y,v_{1m}).
\end{align*}
Therefore, for every two different vertices $x,y$ of $(K_n\times
P_{m})\odot K_1$ we have, $r(x|S)\ne r(y|S)$ and, as a consequence, $S$ is a resolving set for $(K_n\times P_{m})\odot K_1$ of cardinality $n-1$.

On the other hand, by Lemma \ref{remark-kn-pm} and Lemma
\ref{one-pendant-edge} we have $dim((K_{n}\times P_{m})\odot K_1)\ge
n-1$. Hence, for $n\ge 4$, the result follows.
\end{proof}

\section*{Acknowledgements}
This work was partially supported   by the
Spanish Ministry of Education through projects TSI2007-65406-C03-01
 ``E-AEGIS" and Consolider Ingenio 2010 CSD2007-00004 
``ARES''.

\end{document}